\documentclass[11pt,a4paper]{article}

\usepackage{amsmath,amssymb,latexsym,amsfonts,mathrsfs}	
\usepackage{palatino}
\usepackage{amsthm,thmtools}	
\usepackage{hyperref}	
\hypersetup{
	colorlinks=true,	
	linkcolor=[rgb]{0.9, 0.0, 0.0},	
	citecolor=[rgb]{0.0, 0.6, 0.0},	
	filecolor=magenta,	
	urlcolor=blue	
}
\usepackage{authblk}
\usepackage[numbers]{natbib}

\numberwithin{equation}{section}

\newtheorem{defn}[equation]{Def{i}nition}
\newtheorem{prop}[equation]{Proposition}

\newtheorem{theorem}[equation]{Theorem}
\newtheorem{corol}[equation]{Corollary}

\theoremstyle{remark}
\newtheorem{remark}[equation]{Remark}

\theoremstyle{definition}

\providecommand{\keywords}[1]
{
  \small	
  \textbf{\textit{Keywords:}} #1
}

\title{The algebra of invariants of a complete path algebra}

\author[1]{Samuel Quirino}
\affil[1]{Universidade Federal de Minas Gerais, Instituto de Ciências Exatas, \break Departamento de Matemática, Belo Horizonte, MG, Brazil}

\begin{document}

\footnotetext{\textit{Email address:} sasq@ufmg.br \hfill \url{https://orcid.org/0000-0002-7352-3211}}

\maketitle

\begin{abstract}
We prove that the algebra of invariants of a complete path algebra under the action of a homogeneous group of continuous algebra automorphisms is a complete path algebra and preserves finite or tame representation type.
\end{abstract}



\section{Introduction}\label{sec1}
Let $k$ be a field of any characteristic and $Q=(Q_0,Q_1)$ be a quiver. The path algebra $k(Q)$ is isomorphic to the tensor algebra $T_E(M)=\bigoplus_{n=0}^{\infty}M^{\otimes_n}$, where $M^{\otimes_0}=E=\bigoplus_{i\in Q_0} ke_i$ is the semisimple algebra generated by the stationary paths $e_i$, $M^{\otimes_1}=M=\bigoplus_{i,j\in Q_0} VQ_{i,j}$ is the bimodule over $E$, with $VQ_{i,j}$ being the vector space over $k$ generated by all arrows from source $i$ to target $j$, and $M^{\otimes_n}=M\otimes_E (M^{\otimes_{n-1}})$ is the tensor product over $E$ for each $n\geq 2$. 

\begin{remark}
	The object $(Q_0,VQ_{i,j})_{i,j\in Q_0}$ is called $k$-quiver. The tensor algebra $T_E(M)$ is a way of dealing with path algebras when the choice of basis might be a problem. We recommend \cite{IMQ} for details on the relation between $k$-quivers and (complete) path algebras.
\end{remark}

Let $G$ be a group of automorphisms of $T_E(M)$. An element $x \in T_E(M)$ is an invariant of $G$ if $g(x) = x$ for all $g \in G$. The set of all invariants of $G$ is a subalgebra of $T_E(M)$, denoted by $T_E(M)^G$, called the algebra of invariants of $G$. 
An element $x \in T_E(M)$ is homogeneous of degree $n$ if $x\in M^{\otimes_n}$ for some $n\in\mathbb{N}$. An automorphism $g \in G$ is homogeneous if $g(M^{\otimes_n}) = M^{\otimes_n}$ for every $n \in \mathbb{N}$. We say that $G$ is homogeneous if its elements are homogeneous.

If $E = k$ and $M\neq \{0\}$, i.e. the quiver $Q$ has a single vertex and all its arrows are loops, then $T_k(M)$ is a free algebra. In the late 70's, V. K. Kharchenko and D. R. Lane proved that if $G$ is a homogeneous group of algebra automorphisms of the tensor algebra $T_k(M)$, there exists a subspace  $U\subseteq T_k(M)$ generated by homogeneous elements such that $T_k(M)^G \cong T_k(U)$, see \cite[Proposition 1]{Kharchenko78} and \cite[Lemma 1.8]{Lane76}.

In 2016, C. Cibils and E. N. Marcos proved that if $VQ_{i,j}$ is finite dimensional for each $i,j \in Q_0$, i.e. the quiver $Q$ has only finitely many arrows between any pair of vertices, and $G$ is a finite group of homogeneous algebra automorphisms of the tensor algebra $T_E(M)$, then there exists a subspace  $U\subseteq T_E(M)$ generated by homogeneous elements such that $T_E(M)^G \cong T_E(U)$, see \cite[Theorem 4.1]{CMarcos}. Moreover, the representation type of the invariant algebra $T_E(M)^G$ is preserved in case $T_E(M)$ is of finite or tame representation type, see \cite[Theorem 5.16]{CMarcos}.

We show that the above results holds true for complete path algebras.

The complete path algebra is defined as the completion of the path algebra, which is isomorphic to a complete tensor algebra, see \cite[Proposition 8.1]{Simson1}.

From now on, treat $k$ as a discrete topological field. Denote by \break $\widehat{T}_{\Sigma}(V)=\prod_{n=0}^{\infty}V^{\widehat{\otimes}_n}$ the complete tensor algebra of the topologically semisimple pointed pseudocompact algebra $V^{\widehat{\otimes}_0}=\Sigma=\prod_{i\in Q_0} ke_i$ and the pseudocompact $\Sigma$-bimodule $V^{\widehat{\otimes}_1}=V=\prod_{i,j\in Q_0} VQ_{i,j}$, where $V^{\widehat{\otimes}_n} = V \widehat{\otimes}_{\Sigma} V^{\widehat{\otimes}_{n-1}}$ is the complete tensor product for $n\geq 2$.

%
A pseudocompact algebra $A$ is the inverse limit of finite dimensional algebras $\{A_i\}_{i \in I}$, treated as discrete topological algebras, $A = \varprojlim_{i \in I}\, A_i$. A (left) pseudocompact $A$-module is the inverse limit of discrete finite dimensional (left) $A$-modules $\{U_i\}_{i \in I}$, $U = \varprojlim_{i \in I}\, U_i$. If $A$ and $B$ are pseudocompact algebras, then a pseudocompact $A$-$B$-bimodule is a $A$-$B$-bimodule which is a left pseudocompact $A$-module and a right pseudocompact $B$-module.

The inverse limit inherits a topology, which is complete (like any inverse limit of topological sets) and Hausdorff (since it is the inverse limit of discrete topological sets).

\begin{remark}
More often, a pseudocompact algebra is presented by the equivalent definition of a complete Hausdorff topological algebra possessing a fundamental system of neighborhoods of 0 consisting of (two sided) ideals with finite codimension that intersect in 0, see \cite{Brumer}.
\end{remark}

The complete tensor product and the complete tensor algebra are, in some way, the completions of the tensor product and the tensor algebra, respectively, see \cite[\S 2]{Brumer} and \cite[\S 7.5]{Gabriel}. The complete tensor algebra has the following universal property:

Let $A$, $B$ be pseudocompact algebras and $U$ be a pseudocompact $B$-bimodule. Given a continuous algebra homomorphism $\alpha_0:B\to A$ and a continuous $B$-bimodule homomorphism $\alpha_1:U\to A$, with $A$ treated as a $B$-bimodule via $\alpha_0$, then there exists a unique continuous algebra homomorphism $\alpha:\widehat{T}_{B}(U)\to A$ such that $\left.\alpha\right|_{B} = \alpha_0$ and $\left.\alpha\right|_U=\alpha_1$. See \cite[Lemma 2.11]{IM}.

\textbf{Acknowledgements.} This research was developed during the authors Phd Thesis, with advisors Kostiantyn Iusenko and John William MacQuarrie, \cite{Samuel_tese}. The author was partially supported by CAPES -- Finance Code 001.

\section{Invariants of a ring of non-commutative power series rings}
\label{sec:invariants_of_a_power_series_ring}

We must start with Kharchenko and Lane Theorem, since Cibils and Marcos results depend on it. For that, we recall the definitions of weak algorithm and inverse weak algorithm, see \cite[\S2.2 and \S2.9]{Cohn}. Consider $R$ a ring and $\mathbb{N}$ the set of natural numbers including zero.

%
A function $\mu : R \to \mathbb{N} \cup \{-\infty\}$ is a filtration on $R$ if satisfies: 
\begin{enumerate}
	\item $\mu(x) \geqslant 0$ for $x \neq 0$ and $\mu(0) = -\infty$;
	\item $\mu(x-y) \leqslant \textnormal{max}\{\mu(x),\mu(y)\}$;
	\item $\mu(x y)\leqslant \mu(x) + \mu(y)$;
	\item $\mu(1) = 0.$
\end{enumerate}
In case $\mu(x y)= \mu(x) + \mu(y)$, $\mu$ is called a degree function. In general, we say that $\mu(x)$ is the degree of $x$.

We say that $R$ is a graded ring if it can be expressed as the direct sum of abelian groups $R = \bigoplus_{i \in \mathbb{N}} R_i$ such that $R_i R_j \subseteq R_{i+j}$. In this case, $R_0$ is a subring and each $R_i$ is a $R_0$-bimodule.

A graded ring $R = \bigoplus_{i \in \mathbb{N}} R_i$ has a natural degree function $\mu : R \to \mathbb{N} \cup \{-\infty\}$ given by:
\begin{equation*}
	\label{equ:graded_ring_filtration}
	\mu(x) = 
	\begin{cases}
		\textnormal{min}\{n \,:\, x \in \bigcup_{i=0}^{n} R_i\}, & \hbox{if } x \neq 0; \\
		-\infty, & \hbox{if } x = 0.
	\end{cases}
\end{equation*}
In case $x \in R_n$, we say that $x$ is a homogeneous element of degree $n$.

Let $R$ be a ring with filtration $\mu$. For each $n \in \mathbb{N} \cup \{-\infty\}$, denote the set of elements of degree at most $n$ by $R_{(n)} = \{x \in R \,|\, \mu(x) \leqslant n\}$. Then the $R_{(n)}$ are subgroups of the additive group $R$ satisfying:
\begin{enumerate}
	\item $\{0\} = R_{(-\infty)} \subseteq R_{(0)} \subseteq R_{(1)} \subseteq \cdots$;
	\item $\bigcup R_{(n)} = R$;
	\item $R_{(i)} R_{(j)} \subseteq R_{(i+j)}$;
	\item $1 \in R_{(0)}$.
\end{enumerate}

We can form the associated graded ring $\textnormal{gr} (R) = \bigoplus_{n=0}^{\infty} \frac{R_{(i)}}{R_{(i-1)}}$, where the subgroup $R_{(-1)}=R_{(-\infty)}=\{0\}$, thus $\frac{R_{(0)}}{R_{(-1)}}=R_{(0)}$. The natural degree function of the associated graded ring of a filtered ring corresponds to its filtration.

\begin{defn}
	\label{def:weak_algorithm}
	Let $R$ be a ring with filtration $\mu$. A family $\{a_i\}_{i\in I}$ of elements of $R$ is right $\mu$-dependent if some $a_i = 0$ or there exist $b_i \in R$, almost all $0$, such that
	\begin{equation*}
		\mu\bigr(\sum a_i b_i\bigl) < \textnormal{max}\{\mu(a_i)+\mu(b_i)\}.
	\end{equation*}
	An element $a \in R$ is right $\mu$-dependent on $\{a_i\}_{i\in I}$ if $a = 0$ or there exist $b_i \in R$, almost all $0$, such that
	\begin{equation*}
		\mu\bigl(a - \sum a_i b_i\bigr) < \mu(a) \quad \textnormal{and} \quad \mu(a_i) + \mu(b_i) \leqslant \mu(a),\,\forall i\in I.
	\end{equation*}
	The ring $R$ satisfies the $n$-term weak algorithm (with respect to the filtration $\mu$) if for any (right) $\mu$-dependent family with at most $n$ members, say $a_1,\dots,a_m$, ($m \leqslant n$), with $\mu(a_1) \leqslant \dots \leqslant \mu(a_m)$, some $a_i$ is $\mu$-dependent on $a_1, \dots, a_{i-1}$. The ring $R$ satisfies the weak algorithm if it satisfies the $n$-term weak algorithm for all $n \in \mathbb{N}$.
\end{defn}

\begin{prop}[{\cite[Proposition 2.4.2]{Cohn}}]
	Let $R$ be an algebra over a field $k$ with a filtration $\mu$ such that $R_0=k$. Then, R is the free associative (non-commutative) $k$-algebra on a set $X$ and $\mu$ is the formal degree induced from $\mu:X\to\mathbb{N}_{+}$ if and only if $R$ satisfies the weak algorithm.
\end{prop}

This characterization was essential for the proofs provided by Kharchenko and, independetly, Lane for the following theorem:

\begin{theorem}[{\cite[Theorem 6.10.3]{Cohn}}]
	\label{teo:Kharchenko-Lane}
	Let $k\langle X\rangle$ be a free associative (non-commutative) algebra and $G$ a group of automorphisms of $k\langle X\rangle$. If $G$ is homogeneous with respect to the grading on $k\langle X\rangle$ induced by some function $d : X \to \mathbb{N}_{>0}$, then the algebra of invariants of $G$ is free on a set that is homogeneous with respect to $d$.
\end{theorem}

Fortunately, Cohn has developed a similar characterization for the ring of non-commutative power series rings in terms of the inverse weak algorithm, which we present below.

%
A function $\nu : R \to \mathbb{N} \cup \{\infty\}$ is an inverse filtration on $R$ if satisfies: 
\begin{enumerate}
	\item $\nu(x)\in\mathbb{N}$ for $x\neq 0$ and $\nu(0) = \infty$;
	\item $\nu(x-y)\geqslant \textnormal{min}\{\nu(x),\nu(y)\}$;
	\item $\nu(x y)\geqslant \nu(x)+\nu(y)$.
\end{enumerate}
In case $\nu(x y)=\nu(x)+\nu(y)$, $\nu$ is an order function.

Let $R$ be a ring with inverse filtration $\nu$. Denote by $R_{[n]} = \{x \in R \,|\, \nu(x) \geqslant n\}$, which satisfies:
\begin{enumerate}
	\item $R = R_{[0]} \supseteq R_{[1]} \supseteq \cdots$; \label{equ:associated_graded_ring}
	\item $R_{[i]} R_{[j]} \subseteq R_{[i+j]}$;
	\item $\bigcap R_{[n]} = 0$.
\end{enumerate}
	
The associated graded ring is $\textnormal{gr}[R]=\bigoplus_{n=0}^{\infty} \frac{R_{[n]}}{R_{[n+1]}}$. If $x \in R$, $x \neq 0$, and $\nu(x) = n$, denote by $\overline{x} = x + R_{[n+1]} \in \frac{R_{[n]}}{R_{[n+1]}}$.

One could define an inverse weak algorithm in terms of the inverse filtration in a similar manner as the weak algorithm was defined. We prefer the following equivalent definition:

\begin{defn}
	\label{def:inverse_weak_algorithm}
	Let $R$ be a ring with inverse filtration $\nu$. $R$ satisfies the ($n$-term) inverse weak algorithm if the associated graded ring $\textnormal{gr}[R]$ satisfies the ($n$-term) weak algorithm (with respect to its natural degree function).
\end{defn}

If $R$ is an inversely filtered ring, then it is a topological ring with $\{R_{[i]}\,|\,i\in \mathbb{N}\}$ being its neighborhood base at 0. Denote by $\widehat{R}$ its completion. There exists a natural embedding $R\to\widehat{R}$, which respects the (inverse) filtration. If this embedding is an isomorphism, we say that $R$ is complete.

\begin{prop}[{\cite[Proposition 2.9.8]{Cohn}}]
	\label{pro:power_series_ring_characterization}
	Let $A$ be a complete inversely filtered algebra such that $\frac{A}{A_{[1]}} = k$. The algebra $A$ is a ring of non-commutative power series rings 
	if, and only if, $A$ satisfies the inverse weak algorithm.
\end{prop}

\begin{corol}[{\cite[Corollary 2.9.9]{Cohn}}]
	\label{cor:closed_subalgebra_of_power_series_ring}
	Let $B \subseteq A$ be a closed subalgebra of the ring of non-commutative power series rings $A$. If $B$ satisfies the inverse weak algorithm, then $B$ is a ring of non-commutative power series rings.
\end{corol}

The complete tensor algebra $\widehat{T}_{\Sigma}(V)$ has an order function $\nu : \widehat{T}_{\Sigma}(V)\to \mathbb{N} \cup \{\infty\}$ given by $\nu(0)=\infty$ and, for $x\neq 0$, $\nu(x)=\textnormal{max}\{m\,|\,x \in \prod_{n=m}^{\infty} V^{\widehat{\otimes}_{n}}\}$.

Denote by $\widehat{T}_{\Sigma}(V)_{[n]} = \{x \in \widehat{T}_{\Sigma}(V) \,|\, \nu(x) \geq n\}$ and let 
\begin{equation*}
	gr [\widehat{T}_{\Sigma}(V)] =\bigoplus_{n=0}^{\infty} \frac{\widehat{T}_{\Sigma}(V)_{[n]}}{\widehat{T}_{\Sigma}(V)_{[n+1]}} \cong \bigoplus_{n=0}^{\infty} V^{\widehat{\otimes}_n} 
\end{equation*}
be the associated graded ring of $\widehat{T}_{\Sigma}(V)$. If $\Sigma$ and $V$ are finite dimensional, then $gr[\widehat{T}_{\Sigma}(V)] \cong T_{\Sigma}(V)$.
%

An element $x \in \widehat{T}_{\Sigma}(V)$ is called homogeneous if $x \in V^{\widehat{\otimes}_{n}}$ for some $n \in \mathbb{N}$. We say that a continuous algebra automorphisms $g$ of $\widehat{T}_{\Sigma}(V)$ is homogeneous if $g(V^{\widehat{\otimes}_{n}}) = V^{\widehat{\otimes}_{n}}$ for every $n \in \mathbb{N}$. In this case, $g$ induces an algebra automorphism in the associated graded algebra $gr[\widehat{T}_{\Sigma}(V)]$ given by $g(x_n + \widehat{T}_{\Sigma}(V)_{[n+1]}) = g(x_n) + \widehat{T}_{\Sigma}(V)_{[n+1]}$ for any representative $x_n$ of $\overline{x}_n \in \frac{\widehat{T}_{\Sigma}(V)_{[n]}}{\widehat{T}_{\Sigma}(V)_{[n+1]}}$. Moreover, if $G$ is a homogeneous group acting on $\widehat{T}_{\Sigma}(V)$, then $G$ also acts homogeneously on $gr[\widehat{T}_{\Sigma}(V)]$.

\begin{theorem}
	\label{teo:invariants_power_series_ring}
	Let $\widehat{T}_{k}(V)$ be a ring of non-commutative power series rings and $G$ be a homogeneous group of continuous algebra automorphisms of $\widehat{T}_{k}(V)$. Then, the algebra of invariants of $G$, $\widehat{T}_{k}(V)^G$, is a ring of non-commutative power series rings.
\end{theorem}

\begin{proof}
	In view of Proposition \ref{pro:power_series_ring_characterization} and Corollary \ref{cor:closed_subalgebra_of_power_series_ring}, it is sufficient to show that $\widehat{T}_{k}(V)^G$ is closed and $gr[\widehat{T}_{k}(V)^G]$ satisfies the weak algorithm. 
	
	It is closed because $\widehat{T}_{k}(V)^G = \bigcap_{g\in G} \ker (g - \textnormal{id})$.
	
	Since $gr[\widehat{T}_{k}(V)]$ satisfies the weak algorithm and $gr[\widehat{T}_{k}(V)]_0=k$, by Theorem \ref{teo:Kharchenko-Lane}, the algebra of invariantes $gr[\widehat{T}_{k}(V)]^G$ also satisfies the weak algorithm. It is clear that $gr[\widehat{T}_{k}(V)^G] \subseteq gr[\widehat{T}_{k}(V)]^G$. The other way around follows by choosing homogeneous elements. We give the details.
    
    Let $\overline{x} \in (gr[\widehat{T}_{k}(V)])^G$. Then, $\overline{x}$ can be uniquely expressed as $\overline{x} = \bigl(x_n + \widehat{T}_{k}(V)_{[n+1]}\bigr)_{n \geq 0}$, with $x_n \in V^{\widehat{\otimes}_n}$ and only finitely many of them nonzero. Thus, the element $x = (x_n)_{n \geq 0}$ belongs to $\widehat{T}_{k}(V)^G$ and, consequently, $\overline{x} \in gr[\widehat{T}_{k}(V)^G]$. 
\end{proof}

%
\section{Invariants of a complete path algebra}
\label{sec:invariants_of_a_complete_path_algebra}

With Theorem \ref{teo:invariants_power_series_ring}, we can apply the technics developed in \cite[]{CMarcos} to complete path algebras. 
The results presented in this section follows, with small changes, by \cite[Lemma 3.7, Proposition 3.8, Theorem 3.9 and Theorem 4.1]{CMarcos}.

%
In order to simplify the notation, we shall say that any nonzero vector space $VQ_{i,j}$ is an arrow space $V_a = VQ_{i,j}$ from $i$ to $j$. Thus, the subscript $a$ of $V_a$ means that there exists at least one arrow in $Q_1$ form the vertices $i$ to $j$ in $Q_0$. For any sequence of arrow spaces $V_{a_1}, V_{a_2}, \dots, V_{a_m}$, with $V_{a_n} = VQ_{i_{n-1},i_n}$, the vector space $V_{\omega} = V_{a_m} \widehat{\otimes}_{\Sigma} \cdots \widehat{\otimes}_{\Sigma} V_{a_2} \widehat{\otimes}_{\Sigma} V_{a_1} \subseteq V^{\widehat{\otimes}_m}$ is the space of path $\omega$ from $i_0$ to $i_m$ of length $m$ (in particular, any arrow space is a space of path of length $1$). Any subspace of $V_{\omega}$ which is a space of path is called subpath.

Let $G$ be a homogeneous group of continuous algebra automorphisms of $\widehat{T}_{\Sigma}(V)$. If $G$ is invariant on $\Sigma$, then the pseudocompact $\Sigma$-bimodules $V^{\widehat{\otimes}_{n}}$ are left $kG$-modules for all $n \in \mathbb{N}$, where $kG$ denotes the group algebra of $G$.

A 2-partition of a space of path ${\omega}$ is any two subpaths ${\omega_1}$ and ${\omega_2}$ such that $V_{\omega} = V_{\omega_2} \widehat{\otimes}_{\Sigma} V_{\omega_1}$. Write $\omega = \omega_2 \omega_1$ for short. Define $\varphi_{\omega_1,\omega_2} : V_{\omega_2}^G \widehat{\otimes}_{\Sigma} V_{\omega_1}^G \to V_{\omega}^G$ to be the canonical inclusion. Denote by 
\begin{equation*}
	\varphi_{\omega} = \sum_{\omega_2 \omega_1 = \omega} \varphi_{\omega_1,\omega_2}.
\end{equation*}
The image of $\varphi_{\omega}$ is called space of composite invariants. Any complement of $\textnormal{Im}\,{\varphi_{\omega}}$ in the space of invariants is called a space of irreducible invariants. 

Let us fix a complement and identify it by $V_{\omega,\textnormal{irr}}^G$. If $\omega$ has length $m$ and $p = (m_l,\dots,m_1)$ is any ordered $l$-partition of $m$, denote by
\begin{equation*}
	V_{\omega,p,\textnormal{irr}}^G = V_{\omega_{m_l},\textnormal{irr}}^G \widehat{\otimes}_{\Sigma} \cdots \widehat{\otimes}_{\Sigma} V_{\omega_{m_1},\textnormal{irr}}^G,
\end{equation*}
where $\omega_{m_n}$ are the unique subpath of $\omega$ of length $m_n$ such that $\omega = \omega_{m_l} \cdots \omega_{m_1}$. Consider the canonical map 
\begin{equation*}
	\psi_{\omega}:\bigoplus_p V_{\omega,p,\textnormal{irr}}^G \to V_{\omega}^G,
\end{equation*}
where the direct sum runs through all ordered partitions of length $m$ (the length of $\omega$).

%
\begin{prop}
	The map $\psi_{\omega}$ is surjective.
\end{prop}

\begin{proof}	
	It is defined a natural filtration on the invariants and the proof follows by induction. 
	
	For each ordered partition $p = (m_l, \dots, m_1)$ of $m$ (the length of $\omega$) consider the map:
	\begin{equation*}
		\varphi_{\omega,p}: V_{\omega_{m_l}}^G \widehat{\otimes}_{\Sigma} \cdots \widehat{\otimes}_{\Sigma} V_{\omega_{m_1}}^G \to V_{\omega}^G.
	\end{equation*}
	
	Denote the image of the sum of the maps $\varphi_{\omega,p}$ along all the $l$-partitions of $m$ by $[U_{\omega}^G]^l$, called the space of $l$-composite invariants. Thus the 2-composite invariants are the space of composite invariants. The following filtration holds:
	\begin{equation*}
		0 \subseteq [V_{\omega}^G]^m \subseteq \cdots \subseteq[V_{\omega}^G]^1 = V_{\omega}^G.
	\end{equation*}
	
	Observe that
	\begin{equation*}
		[V_{\omega}^G]^m = V_{a_m}^G \widehat{\otimes}_{\Sigma} \cdots \widehat{\otimes}_{\Sigma} V_{a_1}^G,
	\end{equation*}
	which is in the image of $\psi_{\omega}$. 
	
	Assume that $[V_{\omega}^G]^l \subseteq \textnormal{Im}\,(\psi_{\omega})$. Let $v \in [V_{\omega}^G]^{l-1}$ and suppose that $v$ is obtained from a fixed $l-1$ partition. Since $v \in V_{\omega}^G$, $v$ can be decomposed as a sum of two terms: tensors of irreducible invariants, which belong by definition to $\textnormal{Im}\,(\psi_{\omega})$; or $(l-1)$-tensors which contain at least one composite invariant, so belonging to $[V_{\omega}^G]^l$ which is contained in $\textnormal{Im}\,(\psi_{\omega})$ by hypothesis. 
	The general term is simply a sum of such terms, completing the proof.
\end{proof}

\begin{prop}
	Consider $V_{a_n} = V$ for all $n \in \{1, \dots, m\}$, i.e. $V_{\omega} = V^{\widehat{\otimes}_{m}}$. Then, the map $\psi_{\omega}$ is bijective.
\end{prop}

\begin{proof}	
	Let $\widehat{T}_{k}(V)^G = k \times V^G \times (V \widehat{\otimes}_{\Sigma} V)^G \times \cdots$ be the algebra of invariants of $G$. By Theorem \ref{teo:invariants_power_series_ring}, there exists a homogeneous $k$-subbimodule $U \subseteq \widehat{T}_{k}(V)^G$ such that $\widehat{T}_{k}(U) = \widehat{T}_{k}(V)^G$.
	
	Write $U_n = U \cap (V^{\widehat{\otimes}_{n}})^G$, for each $n \geq 1$. Claim: $U_n$ is a vector space of irreducible invariants for every $n$. For $n = 1$ it is clear since $V^G$ is irreducible. Assume that $U_i$ is a space of irreducible invariants of degree $i$, for every $i < n$. Because the composites (in degree $n$) are sums of (complete tensor) products of irreducible of lower degree, they must be obtained from tensors of the $U_i$'s, for $i < n$. Moreover, since $\widehat{T}_{k}(U)$ is a ring of non-commutative power series rings, the intersection of the composites with $U_n$ is zero. Hence $U_n$ is a space of irreducible invariants.
	
	The isomorphisms $U_i \cong (V^{\widehat{\otimes}_{i}})^G_{\textnormal{irr}}$ and $\bigoplus_p(U_{n_l} \widehat{\otimes}_{\Sigma} \cdots \widehat{\otimes}_{\Sigma} U_{n_1}) \cong (V^{\widehat{\otimes}_{n}})^G$, where $p = (n_l, \dots, n_1)$ runs through all partitions of $n$, gives the desired bijection.
\end{proof}

\begin{prop}
	\label{teo:space_of_irreducibles_invariants}
	The map $\psi_{\omega}$ is bijective for any path $\omega$.
\end{prop}

\begin{proof}	
	Let $V = V_{a_1} \times \cdots \times V_{a_n}$ and $V_{\gamma} = V^{\widehat{\otimes}_{m}}$. Then, $\psi_{\gamma}$ is the direct sum of the maps $\psi_{a_{n_m}\dots a_{n_1}}$ along all the sequences $(n_m,\dots,n_1)$ of integers belonging to $\{1,\dots,m\}$, which is bijective. Hence, all those maps are invertible, in particular the one corresponding to the path $\omega$ (i.e. the sequence $(m, \dots, 1)$).
\end{proof}

%
\begin{theorem}
	\label{teo:invariant_complete_path_algebra}
	Let $G$ be a homogeneous group of continuous algebra automorphism of the complete tensor algebra $\widehat{T}_{\Sigma}(V)$ invariant on $\Sigma$. The algebra of invariants $\widehat{T}_{\Sigma}(V)^G$ is isomorphic to a complete tensor algebra.
\end{theorem}

\begin{proof}	
	We construct a family of subbimodules of $\widehat{T}_{\Sigma}(V)^G$ and show that the complete tensor algebra of the product of such family gives the desired isomorphism.
	
	For each pair of vertices $i,j \in Q_0$, denote by $\Omega_{i,j}$ the set of all space of paths from $i$ to $j$. For each $\omega \in \Omega_{i,j}$ fix a space of irreducible invariants $V_{\omega,\textnormal{irr}}^G$, that is
	\begin{equation*}
		V_{\omega}^G = V_{\omega,\textnormal{irr}}^G \oplus \textnormal{Im}\,(\varphi_{\omega}).
	\end{equation*}
	Let $U_{i,j}=\prod_{\omega\in\Omega_{i,j}} V_{\omega,\textnormal{irr}}^G$ and $U=\prod_{i,j\in Q_0} U_{i,j}$. 
	
	The canonical inclusions $\iota_{\Sigma}:\Sigma\to\widehat{T}_{\Sigma}(V)^G$ and $\iota_{U}:U\to\widehat{T}_{\Sigma}(V)^G$ gives a continuous algebra homomorphism $\alpha:\widehat{T}_{\Sigma}(U)\to\widehat{T}_{\Sigma}(V)^G$ by the universal property of the complete tensor algebra.
	
	By Proposition \ref{teo:space_of_irreducibles_invariants}, for any space of path $\omega$, $V_{\omega}^G$ is either a space of irreducible invariants or is isomorphic to a direct sum of tensor products of spaces of irreducible invariants (for any fixed choice of such complements).
	
	By continuity of the elements of $G$, any invariant limit element of a convergent series of elements in $\widehat{T}_{\Sigma}(V)$ must be the limit of a subseries of invariant elements. Thus, continuity of $\alpha$ implies that such element is the image of a limit element in $\widehat{T}_{\Sigma}(U)$. Therefore, $\alpha$ is an isomorphism.
\end{proof}

\section{Invariants and representation types}
\label{sec:invariants_and_representation_types}

A complete path algebra of a connected quiver is of  tame representation type if, and only if, its quiver is infinite locally Dynkin, or is the underlying graph of an Euclidean graph, or it is of finite representation type and the underlying graph of its quiver is a simply laced Dynkin diagram, c.f. \cite[Theorem 7.22]{Simson}. Hence, a complete path algebra of finite or tame representation type is not a finite dimensional path algebra only if it is of tame representation type and its quiver is an infinite locally Dynkin or is a finite cycle.

For more on quiver representations, see \cite[\S 4]{Gabriel}, \cite{BGP}, \cite{Nazarova} and \cite{DGFKK}.

\begin{theorem}
	Let $Q$ be a quiver and $G$ be a finite group of homogeneous continuous algebra automorphism of the complete path algebra $\widehat{T}_{\Sigma}(V)$ invariant on $\Sigma$. If $\widehat{T}_{\Sigma}(V)$ is of finite or tame representation type, then $\widehat{T}_{\Sigma}(V)^G$ is of finite or tame representation type, respectively. 
\end{theorem}

\begin{proof}
	Without loss of generality, we may consider $Q$ connected.
	
	In case $\widehat{T}_{\Sigma}(V)$ is finite dimensional, this follows from \cite[Theorem 5.16]{CMarcos}. 
	
	We have two cases: $Q$ is an infinite locally Dynkin quiver; or $Q$ is an oriented cycle.
	
	In the first case, any finite subquiver of $Q$, say $Q^f$, is Dynkin and the corresponding tensor algebra $\widehat{T}_{\Sigma^f}(V^f)$ is a finite dimensional path algebra of finite representation type. Since $\widehat{T}_{\Sigma}(V)$ is the inverse limit of $\widehat{T}_{\Sigma^f}(V^f)$, running through all finite subquivers of $Q$, it follows that $\widehat{T}_{\Sigma}(V)^G$ is the inverse limit of $\widehat{T}_{\Sigma^f}(V^f)^G$, which associated quiver is Dynkin, and, therefore, its associated quiver is locally Dynkin. 
	
	If $Q$ is a cycle, we analyze the corresponding quiver of $\widehat{T}_{\Sigma}(V)^G$. Given a vertex $i \in Q_0$ and paths $\omega_1$ and $\omega_2$ starting at $i$, then there exists a path $\omega'$ such that $\omega_1 = \omega' \omega_2$ or $\omega_2 = \omega' \omega_1$. Thus, for each vertex $i \in Q_0$ there is at most one irreducible invariant vector space starting at $i$, which corresponds to an arrow space of $\widehat{T}_{\Sigma}(V)^G$. Analogous argument shows that there is at most one irreducible invariant vector space ending at $j$, for each vertex $j \in Q_0$. Thus, the corresponding quiver of $\widehat{T}_{\Sigma}(V)^G$ inherits the property that each vertex is the source of a unique arrow and the target of a unique arrow. 
	
	Since $G$ is a finite group, say with order $m$, and $Q$ is a finite cycle, say with $n$ vertices, any path of length $mn$ is a cycle and belongs to $\widehat{T}_{\Sigma}(V)^G$. Thus, the associated quiver of $\widehat{T}_{\Sigma}(V)^G$ is the union oriented cycles. Therefore, $\widehat{T}_{\Sigma}(V)^G$ is tame.
\end{proof}

\bibliographystyle{plainnat}
\bibliography{references}

\end{document}